\documentclass[12pt,a4paper,reqno]{amsart}

\usepackage{a4wide}
\usepackage{amssymb,amsmath,amsthm}
\usepackage{graphicx}
\usepackage{float}
\usepackage{colortbl}
\usepackage{enumerate}
\usepackage{upref}
\usepackage[textures,backref,bookmarks=false]{hyperref}

\theoremstyle{plain}
\newtheorem{theorem}{Theorem}
\newtheorem{lemma}{Lemma}[section]

\theoremstyle{definition}
\newtheorem{definition}{Definition}[section]

\newtheorem*{condition}{Condition}

\theoremstyle{remark}
\newtheorem{remark}{Remark}[section]

\def\R{\ensuremath{\mathbb R}}
\def\N{\ensuremath{\mathbb N}}

\def\e{\ensuremath{\text e}}
\def\ev{\ensuremath{E\!V}}

\def\B{\ensuremath{\mathcal B}}
\def\M{\ensuremath{\mathcal M}}

\def\l{{\rm Leb}}
\def\P{\ensuremath{\mathcal P}}

\def\X{\mathcal{X}}
\def\ie{{\em i.e.}, }

\def\dist{\ensuremath{\text{dist}}}

\def\cyl{\text{Z}}

\def\size{\hbar}

\numberwithin{equation}{section}

\begin{document}

\title[Extreme Value Laws in Dynamical Systems for Non-smooth Observations]{Extreme Value Laws in Dynamical Systems for Non-smooth Observations}

\author[A. C. M. Freitas]{Ana Cristina Moreira Freitas}
\address{Ana Cristina Moreira Freitas\\ Centro de Matem\'{a}tica \&
Faculdade de Economia da Universidade do Porto\\ Rua Dr. Roberto Frias \\
4200-464 Porto\\ Portugal} \email{amoreira@fep.up.pt}

\author[J. M. Freitas]{Jorge Milhazes Freitas}
\address{Jorge Milhazes Freitas\\ Centro de Matem\'{a}tica da Universidade do Porto\\ Rua do
Campo Alegre 687\\ 4169-007 Porto\\ Portugal}
\email{jmfreita@fc.up.pt}
\urladdr{http://www.fc.up.pt/pessoas/jmfreita}

\author[M. Todd]{Mike Todd}
\address{Mike Todd\\ Department of Mathematics and Statistics\\
Boston University\\
111 Cummington Street\\
Boston, MA 02215\\
USA \\} \email{mtodd@math.bu.edu}

\urladdr{http://math.bu.edu/people/mtodd/}

\thanks{ACMF was partially supported by FCT grant SFRH/BPD/66174/2009. JMF was partially supported by FCT grant SFRH/BPD/66040/2009. MT was partially supported by FCT grant SFRH/BPD/26521/2006 and NSF grants DMS 0606343 and DMS 0908093.  All three authors were supported by FCT through CMUP and PTDC/MAT/099493/2008.}

\date{\today}


\begin{abstract}
We prove the equivalence between the existence of a non-trivial hitting time statistics law and Extreme Value Laws in the case of dynamical systems with measures which are not absolutely continuous with respect to Lebesgue.  This is a counterpart to the result of the authors in the absolutely continuous case.  Moreover, we prove an equivalent result for returns to dynamically defined cylinders.  This allows us to show that we have Extreme Value Laws for various dynamical systems with equilibrium states with good mixing properties.  In order to achieve these goals we tailor our observables to the form of the measure at hand.
\end{abstract}

\maketitle

\section{Introduction}
Understanding extreme events is important in many fields, for example in the data analysis of climate and financial markets.  These examples can be studied using probabilistic models as well as dynamical systems models.  In this paper we prove the equivalence of some notions of extremal events in probability theory and dynamical systems, namely Extreme Value Laws and Hitting time statistics.
Our focus is on systems which are deterministic, but satisfy many of the statistical limit theorems for random processes asymptotically.  This follows on from \cite{FFT10} in which we developed this theory in a `smoother' situation.  Here we prove similar results for systems with `non-smooth' measures and for a correspondingly wider range of observations.  This gives us access to a huge range of statistical limit theorems for dynamical systems which we give some examples of at the end of the paper.  A lengthy introduction is necessary to fix the various notions we use here.  Our main results are contained in Section~\ref{subsec:main results}.

\subsection{Extreme Value Laws}
Consider a sequence of of random variables (r.v.) $X_0, X_1, \ldots$ and define a new stochastic process $M_1, M_2,\ldots$ given by
\begin{equation}
\label{eq:def-Mn}
M_n=\max\{X_0,\ldots,X_{n-1}\}.
\end{equation}
If $M_n$ has a non-degenerate weak limit law under linear normalisation, i.e., if there exist sequences  $(a_n)_{n\in\N}, (b_n)_{n\in\N}$, such that $a_n(M_n-b_n)$ converges in distribution to a non-degenerate distribution function (d.f.) $H$, then we say we have an \emph{Extreme Value Law} (EVL) for $M_n$.

When $X_0, X_1,\ldots$ is an independent and identically distributed (iid) sequence, the classical Extreme Value Theory asserts that $H$ can only be of one of the three classical EVL known as:
\begin{enumerate}[Type 1:]

\item $\ev_1(y)=\e^{-\e^{-y}}$ for $y\in\R$; this is also known
as the \emph{Gumbel} extreme value distribution (evd).

\item $\ev_2(y)=\e^{-y^{-\alpha}}$, for $y>0$, $\ev_2(y)=0$,
otherwise, where $\alpha>0$ is a parameter; this family of d.f.s is known as the \emph{Fr\'echet} evd

\item $\ev_3(y)=\e^{-(-y)^{\alpha}}$, for $y\leq0$, $\ev_3(y)=1$,
otherwise, where $\alpha>0$ is a parameter; this family of d.f.s is
known as the \emph{Weibull} evd
\end{enumerate}

In this paper the sequence of random variables $X_0,X_1,\ldots$ is generated deterministically by a discrete time dynamical system. To be more precise, consider the system $(\X,\mathcal
B,\mu,f)$, where $\X$ is a $d$-dimensional Riemannian manifold,
$\mathcal B$ is the
Borel $\sigma$-algebra, $f:\X\to\X$ is a measurable map
and $\mu$ is an $f$-invariant probability measure (which means that $\mu(f^{-1}(B))=\mu(B)$, for all $B\in\B$). We consider a Riemannian metric on $\X$
that we denote by `$\dist$' and for any $\zeta\in\X$ and $\delta>0$,
we define $B_{\delta}(\zeta)=\{x\in\X: \dist(x,\zeta)<\delta\}$. Take a
r.v. $\varphi:\X\to\R\cup\{\pm\infty\}$
achieving a global maximum at $\zeta\in \X$ (we allow
$\varphi(\zeta)=+\infty$), consider the stationary stochastic
process $X_0, X_1,\ldots$ given by
\begin{equation}
\label{eq:def-stat-stoch-proc} X_n=\varphi\circ f^n,\quad \mbox{for
each } n\in {\mathbb N}.
\end{equation}
Clearly, $X_0, X_1,\ldots$ defined in this way is not an independent sequence.  However, $f$-invariance of $\mu$ guarantees that this stochastic process is stationary.

\subsection{Hitting/Return Time Statistics}
Consider now a set $A\in\B$ and a new r.v. that we refer to as \emph{first hitting time} to $A$ and denote by $r_A:\X\to\N\cup\{+\infty\}$ where
\begin{equation}
\label{eq:def-hitting time}
r_A(x)=\min\{j\in\N\cup\{+\infty\}:\; f^j(x)\in A\}.
\end{equation}
Given a sequence of sets $\{U_n\}_{n\in \N}$ so that
$\mu(U_n) \to 0$ we define the stochastic process $r_{U_1}, r_{U_2},\ldots$ If under suitable normalisation $r_{U_n}$ converges in distribution to some non-degenerate d.f. $G$ we say that the system has \emph{Hitting Time Statistics} (HTS) for $\{U_n\}_{n\in\N}$. For systems with `good mixing properties', $G$ is the standard exponential d.f., in which case, we say that we have \emph{exponential HTS}.

We say that the system has HTS $G$ to balls at $\zeta$ if for any sequence $(\delta_n)_{n\in \N}\subset \R^+$ such that $\delta_n\to 0$ as $n\to \infty$ we have HTS $G$ for $(U_n)_n=(B_{\delta_n}(\zeta))_n$.

Let $\P_0$ denote a partition of $\X$. We define the corresponding pullback partition $\P_n=\bigvee_{i=0}^{n-1} f^{-i}(\P_0)$, where $\vee$ denotes the join of partitions. We refer to the elements of the partition $\P_n$ as cylinders of order $n$. For every $\zeta\in\X$, we denote by $\cyl_n[\zeta]$ the cylinder of order $n$ that contains $\zeta$. For some $\zeta\in \X$ this cylinder may not be unique, but we can make an arbitrary choice, so that $\cyl_n[\zeta]$ is well defined. We say that the system has HTS $G$ to cylinders at $\zeta$ if  we have HTS $G$ for $U_n=\cyl_n(\zeta)$.

Let $\mu_A$ denote the conditional measure on $A\in\B$, \ie $\mu_A:=\frac{\mu|_A}{\mu(A)}$. Instead of starting  somewhere in the whole space $\X$, we may want to start in $U_n$ and study the fluctuations of the normalised return time to $U_n$ as $n$ goes to infinity, \ie for each $n$, we look at the random variables $r_{U_n}$ as being defined in the probability space $(U_n,\B\cap U_n, \mu_{U_n})$ and wonder if, under some normalisation, they converge in distribution to some non-degenerate d.f. $\tilde G$, in which case, we say that the system has \emph{Return Time Statistics} (RTS) for $\{U_n\}_{n\in\N}$. The existence of exponential HTS is equivalent to the existence of exponential RTS. In fact, according to \cite{HLV05}, a system has HTS $G$ if and only if it has RTS $\tilde G$ with $G(t)=\int_0^t(1-\tilde G(s))\,ds $.

In \cite{FFT10}, the authors established a relation between the existence of HTS for balls and EVL for the stochastic processes defined in \eqref{eq:def-Mn} arising from the stationary sequence of random variables given by \eqref{eq:def-stat-stoch-proc}. This link was proved in the case where the invariant probability measure $\mu$ is absolutely continuous with respect to Lebesgue measure (an \emph{acip}) and allowed us to study HTS with tools of EVL and vice-versa in that setting. This connection was applied with success to non-uniformly hyperbolic systems both in the unidimensional and multidimensional cases. The main goal of this paper is to broaden the applications scenario by establishing the connection between HTS and EVL without the restraint of the invariant probability measure being absolutely continuous with respect to Lebesgue or, in other words, $\mu$ is not an acip.
This lets us study the cases where the invariant measure is an equilibrium state, for example, but not an acip.
We accomplish this generalisation by tailoring the observable to the measure $\mu$.  This yields a non-smooth observable $\varphi$.  A precise statement of our results are contained in Section~\ref{subsec:main results}.

Early results regarding laws of rare events for dynamical systems \cite{P91,H93}, showed the occurrence of exponential RTS and HTS for returns/hits to dynamical cylinders rather than to balls.
Indeed, the theory of rare events still has its widest application for this kind of returns/hits.
Therefore, another purpose of this paper is to interpret HTS for cylinders in the context of EVL. Again, this will be achieved by adapting the observable $\varphi$ to provide the desired connection, which in this case will imply adjusting the notion of EVL to apply when the convergence of $a_n(M_n-b_n)$ occurs only for particular subsequences.  Again, to see a precise statement of this theorem see Section~\ref{subsec:main results}.

\section{Normalising sequences for EVL and HTS/RTS}

If $\zeta$ belongs to the support of $\mu$ and the probability measure $\mu$ is ergodic then $M_n\to \varphi (\zeta)$ almost surely. Hence, to understand this convergence more fully, one has to find normalising sequences  $(a_n)_{n\in\N}\subset \R^+$ and
$(b_n)_{n\in\N}\subset\R$ such that
\begin{equation}
\label{eq:def-EVL} \mu\left(\{x:a_n(M_n-b_n)\le
y\}\right)=\mu\left(\{x:M_n\le u_n\}\right)\to H(y),
\end{equation}
for some non-degenerate d.f. $H$, as
$n\to\infty$.  Here the level $u_n$ is linear on $y$ \begin{equation}\label{def:level-un}u_n:=
u_n(y)=\frac{y}{a_n}+b_n\end{equation}
and must be such that $u_n\to\varphi(\zeta)$, as $n\to\infty$ in order to get a limiting law.
We refer to an event $\{X_j>u_n\}$ as an \emph{exceedance}, at time $j$, of level $u_n$ and it is clear that $\mu(X_j>u_n)\to 0$ as $n\to\infty$.

Under some mixing conditions it is possible to reduce the study of EVL for stationary stochastic processes to that of iid sequences. Hence, we motivate the choice of the normalising sequence of levels $(u_n)_{_n\in\N}$ by following the procedure in the iid case. First, to the stationary sequence $X_0,X_1,\ldots$ we associate an iid sequence of r.v.s $Y_0, Y_1,\ldots$ such that each $Y_j$ has the same d.f. of any $X_i$, for all $j,i\in\N$, and define
\begin{equation}\label{eq:def-max-iid}\hat
M_n:=\max\{Y_0,\ldots,Y_{n-1}\}.
\end{equation}
In order to compute the rate at which $u_n\to \varphi(\zeta)$ so that $\hat M_n$ is normalised, observe that by independence $\mu(\hat M_n\leq u_n)=\left(\mu(Y_0\leq u_n)\right)^n=\left(\mu(X_0\leq u_n\right)^n$ which gives
\begin{align*}
\log(\mu(\hat M_{n}\leq u_{n}))&=\log(\mu( X_0\leq u_{n})^{n})=n\log(\mu( X_0\leq u_{n}))\\
&=n\log(1-\mu( X_0>u_{n}))\sim -n\mu( X_0>u_{n})
\end{align*}
Throughout this paper the notation $A_n\sim B_n$ means that
$\lim_{n\to \infty} \frac{A_n}{B_n}=1$. Hence, if there exists some $0\leq\tau\leq\infty$ such that
\begin{equation}
\label{eq:un}
  n\mu(X_0>u_n)\to \tau,\;\mbox{ as $n\to\infty$,}
\end{equation}
then
\begin{equation}
\label{eq:iid}
  \mu(\hat M_{n}\leq u_{n})\to \e^{-\tau},\;\mbox{ as $n\to\infty$,}
\end{equation}
and the reciprocal is also true. Observe that $\tau$ depends on $y$ through $u_n$ and, in fact, depending on
the type of limit law that applies, we have that $\tau=\tau(y)$ is
of one of the following three types:
\begin{equation}\tau_1(y)=\e^{-y} \text{ for } y \in
{\mathbb R},\quad \tau_2(y)=y^{-\alpha} \text{ for } y>0\quad \text{ and }\quad
\tau_3(y)=(-y)^{\alpha} \text{ for } y\leq0.
\label{eq:tau}
\end{equation}

This takes care of the normalising sequences and so, more explicitly, we say:
\begin{definition}
\label{def:EVL}
we have an EVL $H$ for $M_n$ if \eqref{eq:def-EVL} holds for normalising sequences $a_n$, $b_n$ such that $u_n$ defined in \eqref{def:level-un} satisfies \eqref{eq:un} and $H(y)=\tilde H(\tau(y))$, where $\tau(y)$ is as in \eqref{eq:tau} and $\tilde H$ is some real valued function.
\end{definition}

Regarding normalising sequences to obtain HTS/RTS, we recall Kac's Lemma, which states that the expected value of $r_A$ with respect to $\mu_A$ is  $\int_A r_A~d\mu_A =1/\mu(A)$.  So in studying the fluctuations of $r_A$ on $A$, the relevant normalising factor should be $1/\mu(A)$.
\begin{definition}
\label{def:HTS/RTS}
Given a sequence of sets $(U_n)_{n\in \N}$ so that
$\mu(U_n) \to 0$, the system has \emph{HTS}
$G$ for $(U_n)_{n\in \N}$ if for all $t\ge 0$
\begin{equation}\label{eq:def-HTS-law}
\mu\left(r_{U_n}\geq\frac t{\mu(U_n)}\right)\to G(t) \;\mbox{ as $n\to\infty$,}
\end{equation}
and the system has \emph{RTS}
$\tilde G$ for $(U_n)_{n\in \N}$ if for all $t\ge 0$
\begin{equation}\label{eq:def-RTS-law}
\mu_{U_n}\left(r_{U_n}\geq\frac t{\mu(U_n)}\right)\to\tilde G (t)\;\mbox{ as $n\to\infty$}.
\end{equation}
\end{definition}

\section{Existence of laws of rare events }

In order to show directly the existence of EVL for dynamical systems, we refer to \cite{FF08a} where the general strategy is to prove that $X_0,X_1,\ldots$ satisfies
some mixing conditions which allow the reduction to the iid case. Following \cite{LLR83} we refer to these
conditions as $D(u_n)$ and $D'(u_n)$, where $u_n$ is the sequence of
thresholds appearing in \eqref{eq:def-EVL}. Both conditions impose
some sort of independence but while $D(u_n)$ acts on
 the long range, $D'(u_n)$ is a short
range requirement.

\begin{condition}[$D(u_n)$]\label{cond:D2}We say that $D(u_n)$ holds
for the sequence $X_0,X_1,\ldots$ if for any integers $\ell,t$
and $n$
\[ \left|\mu\left(\{X_0>u_n\}\cap
  \{\max\{X_{t},\ldots,X_{t+\ell-1}\}\leq u_n\}\right)-\mu(\{X_0>u_n\})
  \mu(\{M_{\ell}\leq u_n\})\right|\leq \gamma(n,t),
\]
where $\gamma(n,t)$ is nonincreasing in $t$ for each $n$ and
$n\gamma(n,t_n)\to0$ as $n\rightarrow\infty$ for some sequence
$t_n=o(n)$.
\end{condition}
This condition follows immediately from sufficiently fast
decay of correlations for observables which are of bounded variation or
H\"older continuous (see \cite[Section~2]{FF08a} and \cite[Lemma~6.1]{FFT10}).

By \eqref{eq:un}, the sequence $u_n$ is such that the average number of exceedances in
the time interval $\{0,\ldots,\lfloor n/k\rfloor\}$ is approximately $\tau/k$,
which goes to zero as $k\rightarrow\infty$. However, the exceedances
may have a tendency to be concentrated in the time period following
the first exceedance at time $0$. To avoid this we introduce:

\begin{condition}[$D'(u_n)$]\label{cond:D'un} We say that $D'(u_n)$
holds for the sequence $X_0,X_1,\ldots$ if
\begin{equation}
\label{eq:D'un} \lim_{k\rightarrow\infty}
\limsup_{n\rightarrow\infty}\,n\sum_{j=1}^{\lfloor n/k\rfloor}
\mu(\{X_0>u_n\}\cap
\{X_j>u_n\})=0.
\end{equation}
\end{condition}

This guarantees that the
exceedances should appear scattered through the time period
$\{0,\ldots,n-1\}$.

The main result in \cite[Theorem 1]{FF08a} states that if $D(u_n)$ and
$D'(u_n)$ hold for the process $X_0, X_1,\ldots$ and for a sequence of levels satisfying \eqref{eq:un},
then the following limits exist, and
\begin{equation}
  \label{eq:rel-hatMn-Mn}
  \lim_{n\to\infty}\mu(\hat M_n\leq u_n)=
  \lim_{n\to\infty}\mu(M_n\leq u_n).
\end{equation}

The existence of EVLs for dynamical systems is a recent topic and have been proved for non-uniformly hyperbolic systems in the pioneer paper \cite{C01}. Since then other results followed in \cite{H03,FF08,FF08a,FFT10,HNT}.

On the other hand, the theory of HTS/RTS laws is now a well developed theory, applied first to cylinders and hyperbolic dynamics, and then extended to balls and also to non-uniformly hyperbolic systems. We refer to \cite{C00} and \cite{S09} for very nice reviews as well as plenty of references on the subject.  (See also \cite{AG01}, where the focus is more towards a finer analysis of uniformly hyperbolic systems.)  Several different approaches have been used to prove HTS/RTS: from the analysis of adapted Perron-Frobenius operators in \cite{H93}, the use of inducing schemes in \cite{BST03,BV03,BT09a}, to the relation between recurrence rates and dimension as explained in \cite[Section~4]{S09}.  We would like to give particular mention to \cite{HSV99} in which general mixing conditions were introduced, under which exponential HTS/RTS hold. These conditions are related to quantities denoted by $a_N(U)$ and $b_N(U)$ in \cite[Lemma~2.4]{HSV99}.  It turns out that $D'(u_n)$ is closely related to $a_N(U)$ in the sense that both require some sort of short range independence while $D(u_n)$ is linked to $b_N(U)$ and imposes some mixing type of behaviour.
\section{The choice of observables.}

We assume that the observable $\varphi:\X\to\R\cup\{+\infty\}$ is of
the form
\begin{equation}
\label{eq:observable-form} \varphi(x)=g\left(\mu\big(B_{\dist(x,\zeta)}(\zeta)\big)\right),
\end{equation} where $\zeta$ is a chosen point in the
phase space $\X$ and the function $g:[0,+\infty)\rightarrow {\mathbb
R\cup\{+\infty\}}$ is such that $0$ is a global maximum ($g(0)$ may
be $+\infty$); $g$ is a strictly decreasing bijection $g:V \to W$
in a neighbourhood $V$ of
$0$; and has one of the
following three types of behaviour:
\begin{enumerate}[Type 1:]
\item there exists some strictly positive function
$p:W\to\R$ such that for all $y\in\R$
\begin{equation}\label{eq:def-g1}\displaystyle \lim_{s\to
g_1(0)}\frac{g_1^{-1}(s+yp(s))}{g_1^{-1}(s)}=\e^{-y};
\end{equation}
\item $g_2(0)=+\infty$ and there exists $\beta>0$ such that
for all $y>0$
\begin{equation}\label{eq:def-g2}\displaystyle \lim_{s\to+\infty}
\frac{g_2^{-1}(sy)}{g_2^{-1}(s)}=y^{-\beta};\end{equation}
\item $g_3(0)=D<+\infty$ and there exists $\gamma>0$ such
that for all $y>0$
\begin{equation}\label{eq:def-g3}\lim_{s\to0}
\frac{g_3^{-1}(D-sy)}{g_3^{-1}(D-s)}= y^\gamma.
\end{equation}
\end{enumerate}

Examples of each one of the three types are as follows:
$g_1(x)=-\log x$ (in this case \eqref{eq:def-g1} is easily verified
with $p\equiv1$), $g_2(x)=x^{-1/\alpha}$ for some $\alpha>0$ (condition
\eqref{eq:def-g2} is verified with $\beta=\alpha$) and
$g_3(x)=D-x^{1/\alpha}$ for some $D\in\R$ and $\alpha> 0$ (condition
\eqref{eq:def-g3} is verified with $\gamma=\alpha$).

In \cite{FFT10} we assumed that $\varphi(x)=g\big(\dist(x,\zeta)\big)$. Since the invariant measure there was an acip, using Lebesgue's differentiation theorem, we could write $\mu(B_\eta (\zeta))\sim \rho (\zeta) \l\left(B_\eta(\zeta)\right)$, where we assume that $\rho(\zeta)=\frac{d\mu}{\l}(\zeta)>0$ and Lebesgue's differentiation theorem applies to $\zeta$. In here, since $\mu$ may not be an acip the function $\size$ defined for small $\eta\geq 0$ and given by
\begin{equation}
\label{eq:def-h}
\size(\eta)=\mu(B_\eta (\zeta))
\end{equation}
may not be absolutely continuous. However, we require that $\size$ is continuous on $\eta$. For example, if $\X$ is an interval and $\mu$ a Borel probability with no atoms,\ie points with positive $\mu$ measure, then $\size$ is continuous. 
One of our applications is to equilibrium states, which we explain in Section~\ref{sec:eq}.  A major difference here is that although $g$ is invertible in a small neighbourhood of $0$, the function $\size$ does not have to be. This means that, in contrast with \cite{FFT10}, the observable $\varphi$, as a function of the distance to $\zeta$, may not be invertible in any small neighbourhood of $\zeta$.

For that reason, we now define
\begin{equation}
\label{eq:l}\ell(\gamma):=\inf\{\eta>0:\mu(B_\eta(\zeta))=\gamma\}.
\end{equation}
In particular, we have
\begin{equation}
\label{eq:l-property}
\mu\left(B_{\ell(\gamma)}(\zeta)\right)=\gamma.
\end{equation}

\section{Limit laws for cylinders}
\label{sec:cyl}

In order to make the connection between HTS/RTS and EVL for cylinders we make a suitable choice of the observable $\phi$, which, in this case, we set it be of the form
\begin{equation}
\label{eq:def-observable-cylinders}
\varphi=g_i\circ \psi,
\end{equation}
where $g_i$ is one of the three forms given above and $\psi(x):= \mu(\cyl_n[\zeta])$ where $n$ is maximal such that $x\in \cyl_n[\zeta]$.

The highly irregular behaviour of $\psi$ leads us to an adjustment of the definition of EVL, which we will refer to as a \emph{cylinder EVL}. The problem arises with the possible nonexistence of a sequence of levels $u_n$ such that \eqref{eq:un} holds.

To illustrate the problem and to motivate our definition of EVL for cylinders we consider the so-called \emph{full tent map} $f:[0,1]\to[0,1]$ given by $$f(x)=1-|2x-1|,$$ with the partition $\mathcal P_0=\left\{\left[0,\frac12\right],\left(\frac12,1\right]\right\}$.  This is the situation considered in \cite{H03} and, in many aspects, is as good as it gets. For definiteness take $\zeta=1$, $g_2(x)=1/x$ and stipulate that $g_2(0)=+\infty$. In this case, it is easy to check that Lebesgue measure is invariant hence we assume that $\mu$ stands for Lebesgue measure on $[0,1]$. Besides, for every $j\in\N$,  we have $Z_j(\zeta)=(1-2^{-j},1]$ and $\mu(Z_j(\zeta))=2^{-j}$. Let $F$ denote the d.f. of $X_0$, \ie $F(y)=\mu(X_0\leq y)$. Observe that $F$ is discontinuous. In fact, at every $y_j=2^j$, with $j\in\N$, the d.f. $F$ has a jump of size $2^{-j}$. These jumps at $y_j$ are too big when compared to $1-F(y_j)$ and make it impossible to find a sequence $u_n$ such that \eqref{eq:un} holds for some $\tau>0$. The natural candidate here would be to take $u_n=2^{[\log_2 n]}$. However, $n\mu(X_0>u_n)=n(1-F(u_n))=\frac n{2^{[\log_2 n]}}$ oscillates too much to have a limit.  This phenomenon also occurs for general choices of $\zeta$ and for general dynamical systems.

Also, the Shannon-McMillan-Breimann Theorem says that if the metric entropy $h_\mu$ is positive, then for $\mu$-a.e. $\zeta$, the cylinders $\cyl_n[\zeta]\in \P_n$ satisfy
$$\lim_{n\to \infty}\frac{-\log\mu(\cyl_n[\zeta])}n \to h_\mu.$$
This means that even for `well behaved' systems such as the full tent map, $n\mu(X_0>u_n)$ can fluctuate wildly since $\mu(\cyl_n[\zeta])\sim\e^{-h_\mu n}$, which creates jumps in the tail of the d.f. $F$ which are too big when compared to the value of the tail of $F$ at the jumps.  Indeed in the special case of the full tent map, we also have
\begin{equation}
\frac{\mu(\cyl_n[\zeta])}{e^{-nh_\mu}}=1.\label{eq:tent no fluc}
\end{equation}
(More generally, for more complicated measures and systems, this quantity also fluctuates wildly in $n$, see Remark~\ref{rmk:Gibbs fluc} for a note on the situation for Gibbs measures.)

A possible solution for this issue is to take a subsequence of the time $n$, which we denote by $(\omega_n)_{n\in\N}$ and  such that
\begin{equation}
  \label{eq:wn-condition}
  \omega_n \mu(X_0>u_n)\xrightarrow[n\to\infty]{}\tau>0.
  \end{equation}
So, for the full tent map, for any $\tau>0$, one could take for example:
\begin{equation}
\label{eq:tent-un-wn}
\omega_n=[\tau 2^n]\qquad\mbox{and}\qquad u_n=2^n,
\end{equation} and we would get that $\omega_n\mu(X_0>u_n)= [\tau 2^n] 2^{-n}$ converges to $\tau>0$.

\begin{remark}
The choice of $(\omega_n)_n$ for the full tent map is of a particularly nice form: $[\tau e^{\alpha n}]$.  This follows since all $n$-cylinders have equal measure, as in \eqref{eq:tent no fluc}. This is far from the general situation, in which we would expect $\frac{\mu(\cyl_n[\zeta])}{e^{-nh_\mu}}$ to fluctuate wildly.  For this reason, there is no general way of choosing $\omega_n$ to be of the form $[\tau e^{\alpha n}]$ for some fixed $\alpha$ which can depend on $\zeta$.
\end{remark}

Moreover, as for the case of balls:
\begin{lemma}
If $\omega_n \mu(X_0>u_n)\xrightarrow[n\to\infty]{}\tau\geq 0$, then $\lim_{n\to \infty}\mu\left(\hat M_{\omega_n} \le u_n\right) \to e^{-\tau}.$
\label{lem:iid-cyl}
\end{lemma}

\begin{proof}
Recall that
$\mu\left(\hat M_{\omega_n} \le u_n\right) = \left(1-\mu\left(X>u_n\right)\right)^{\omega_n}.$
Since $$\log \left(1-\mu\left(X>u_n\right)\right)^{\omega_n} \sim -\omega_n\mu\left(X>u_n\right),$$
the lemma follows.
\end{proof}

In the case of EVLs for observations compatible with balls, as in the standard EVT setting, we took samples of $M_n$ at times $n=1, 2, \dots$ and so on.  This fitted in with the natural scaling given by the measures of the balls, which were of order $1/n$.  When our observables are compatible with cylinders, the time scale $\omega_1, \omega_2, \ldots$ should be the reciprocal of the measure of the cylinders, which on average decay exponentially fast.

However, a new complication emerges with this strategy that we already bypassed in the full tent map case with our choice of $\omega_n$ and $u_n$ in \eqref{eq:tent-un-wn}. By definition EVLs are limit laws for the maxima under linear normalisation which means that $u_n$ is of the form \eqref{def:level-un} which depends on a factor $y$. In the full tent map case, if we were to choose instead $\omega_n=2^n$ and $u_n=2^n y$ (which is the typical choice for an observable $g_2$), using Lemma~\ref{lem:iid-cyl}, we would be led to the limit law $H(y)=\e^{-2^{-[\log_2y]}}$ which is not continuous and not one of the classical $\ev_i$, $i=1,2,3$. In fact, since $\omega_n$ is not linear in $n$, we cannot obtain a max stable law (see Remark~\ref{rem:max-stable} for definition) with this last type of normalisation.

Hence, while previously we built the dependence of $\mu(X>u_n)$ on $\tau$ into $(u_n)_{n\in\N}$,  in this setting it is necessary, if our results are to hold for general dynamical systems, to build the dependence on $\tau$ into the time scale. For every $n\in\N$, $\tau\geq 0$, let $u_n$ be such that
\begin{equation}
\label{eq:un-cylinders}
\{X_0>u_n\}=\cyl_n[\zeta]
\end{equation} and  set
\begin{equation}
\label{eq:wn-definition}
\omega_n=\omega_n(\tau)=[\tau \left(\mu(X_0>u_n)\right)^{-1}].
\end{equation}

Finally, we say that we have a \emph{cylinder EVL} $H$ for the maximum if  for any sequence $(u_n)_{n\in\N}$ such that \eqref{eq:un-cylinders} holds and for $\omega_n$ defined in \eqref{eq:wn-definition}, the limit \eqref{eq:wn-condition} holds and
\[
\mu\left(M_{\omega_n}\le u_n\right)\to H(\tau),
\]
for some non-degenerate d.f. $H$, as
$n\to\infty$.

It is clear that the existence of an EVL for balls is a rather stronger statement then the existence of a cylinder EVL for a particular system since the later only requires convergence on certain suitable subsequences. We mention that Haiman obtained an exponential cylinder EVL for the full tent map. In his paper $\zeta=1/2$, $g_2(x)=1-|2x-1|$ and \cite[Theorem~2]{H03} states that
$$\lim_{n\to \infty}\mu\left(M_{[\tau 2^n]}\le 1-2^{-n}\right) =e^{-\tau}.$$
However, it is possible to show that the full tent map admits an actual EVL for balls centred on the vertex \cite{F09}.

In order to prove the existence of an exponential cylinder EVL it is enough to check conditions $D$ and $D'$ on the subsequence $\omega_n$: let $u_n$ and $\omega_n$ be defined as in \eqref{eq:un-cylinders} and \eqref{eq:wn-definition}, respectively, and consider the conditions:
\begin{condition}[$D(u_n,\omega_n)$]\label{cond:D-cyl}We say that $D(u_n,\omega_n)$ holds
for the sequence $X_0,X_1,\ldots$ if for any integers $\ell,t$
and $n$
\[ \left|\mu\left(\{X_0>u_n\}\cap
  \{\max\{X_{t},\ldots,X_{t+\ell-1}\}\leq u_n\}\right)-\mu(\{X_0>u_n\})
  \mu(\{M_{\ell}\leq u_n\})\right|\leq \gamma(n,t),
\]
where $\gamma(n,t)$ is nonincreasing in $t$ for each $n$ and
$\omega_n\gamma(n,t_n)\to0$ as $n\rightarrow\infty$ for some sequence
$t_n=o(\omega_n)$;
\end{condition}
\begin{condition}[$D'(u_n,\omega_n)$]\label{cond:D'un-cyl} We say that $D'(u_n,\omega_n)$
holds for the sequence $X_0,X_1,\ldots$ if
\begin{equation}
\label{eq:D'un-cyl} \lim_{k\rightarrow\infty}
\limsup_{n\rightarrow\infty}\,\omega_n\sum_{j=1}^{\lfloor \omega_n/k\rfloor}
\mu(\{X_0>u_n\}\cap
\{X_j>u_n\})=0.
\end{equation}
\end{condition}

If $D(u_n,\omega_n)$ and $D'(u_n,\omega_n)$ hold then
\begin{equation}\label{eq:cylinder-exp-EVL}
\mu\left(M_{\omega_n}\le u_n\right)\to \e^{-\tau},\quad \mbox{as $n\to\infty$.}
\end{equation}
The proof of this statement follows from Lemma~\ref{lem:iid-cyl} and a straightforward adaption of the argument in the proof of \cite[Theorem 1]{FF08a}

\begin{remark}
\label{rem:max-stable}
We say that a nondegenerate d.f.s $H$ is max-stable if, for each $n=2,3,\ldots$, there are constants $a_n>0$ and $b_n$ such that $H^n(a_n x+b_n)=H(x)$. A nondegenerate function $H$ is \emph{max-stable} if and only if there is a sequence $(F_n)_{n\in\N}$ of d.f.s and constants $a_n>0$ and $b_n$ such that $F_n(a_{nk}^{-1}x+b_{nk})\rightarrow H^{1/k}(x)$ as $n\rightarrow\infty$, for each $k=1,2,\ldots$. As a consequence of this result we can see that the class of nondegenerate d.f. which appear as limit laws in (\ref{eq:def-EVL}) coincides with the class of max-stable d.f.s.
\end{remark}

\begin{remark}
We can also apply the above theory to so-called `dynamical balls', also known as `Bowen balls'.  For a dynamical system $f:\X\to \X$, a point $\zeta\in\X$ and $\varepsilon>0$, the set $B_n(\zeta, \varepsilon):=\{y\in \X:d(f^j(x), f^j(y))<\varepsilon \text{ for every } 1\le j\le n\}$ is an $(n,\varepsilon)$-dynamical ball around $\zeta$.  Recurrence for this type of ball was studied, for example, in \cite{V09}.
\end{remark}

\section{Main Results}
\label{subsec:main results}

Our first main result, which obtains EVLs from HTS for balls, is the following.

\begin{theorem}
\label{thm:HTS-implies-EVL} Let $(\X,\mathcal B, \mu,f)$ be a
dynamical system, $\zeta\in\X$ be in the support of $\mu$ and assume that $\mu$ is such that the function $\size$ defined on \ref{eq:def-h} is continuous.
\begin{itemize}
\item If we have HTS $G$ to balls centred on $\zeta\in\X$,
then we have an EVL $H$ for $M_n$ which applies to the observables
\eqref{eq:observable-form} achieving a maximum at $\zeta$, where $H(y)=G(\tau(y))$ and $\tau$ is of one of the forms $\tau_i$ given in \eqref{eq:tau}.

\item If we have exponential HTS ($G(t)=\e^{-t}$)
to balls at 
$\zeta\in\X$, then we have an EVL for $M_n$ which coincides with that of $\hat M_n$ (meaning that \eqref{eq:rel-hatMn-Mn} holds). In particular, this EVL must be one of the 3 classical types. Moreover, if $g$ is of type $g_i$,  for some $i\in\{1,2,3\}$, then we have an
EVL for $M_n$ of type $\ev_i$.
\end{itemize}
\end{theorem}

Now, we state a result in the other direction, \ie we show how to  get HTS from EVLs for balls.

\begin{theorem}
\label{thm:EVL=>HTS} Let $(\X,\mathcal B, \mu,f)$ be a
dynamical system, $\zeta\in\X$ be in the support of $\mu$ and assume that $\mu$ is such that the function $\size$ defined in \ref{eq:def-h} is continuous.
\begin{itemize}

\item If we have an EVL $H$  for $M_n$ which applies to the observables
\eqref{eq:observable-form} achieving a maximum at $\zeta\in\X$ then
we have HTS $G$ to balls at $\zeta$, where $H(y)=G(\tau (y))$ and $\tau$ is of one of the forms $\tau_i$ given in \eqref{eq:tau}.

\item If we have an EVL for $M_n$ which coincides with that of $\hat M_n$, then
we have exponential HTS ($G(t)=\e^{-t}$) to balls at $\zeta$.
\end{itemize}
\end{theorem}

Finally, we state a result relating cylinder EVLs and HTS for cylinders.

\begin{theorem}
\label{thm:cylinderEVL-HTS}
Let $(\X,\mathcal B, \mu,f)$ be a
dynamical system, $\zeta\in\X$  be in the support of $\mu$. We have a cylinder EVL $H$ for the maximum, where the observable is given by \ref{eq:def-observable-cylinders}, if and only if we have HTS $H$ to cylinders which is to say that
\[\lim_{n\to \infty}
\mu\left(M_{\omega_n} \le u_n\right) = H(t)=\lim_{n\to \infty}\mu\left(r_{\cyl_n[\zeta]}\geq\frac t{\mu(\cyl_n[\zeta])}\right),
\]
for the sequences $(u_n)_{n\in\N},(\omega_n)_{n\in\N}$ such that \eqref{eq:un-cylinders} and \eqref{eq:wn-definition} hold.

Moreover, if $H(t)=\e^{-t}$ then we may replace $M_{\omega_n}$ by $\hat M_{\omega_n}$ and/or $\mu$ by $\mu_{\cyl_n[\zeta]}$ in the equation above.

\end{theorem}

\section{Proofs}
\label{sec:proofs HTS EVL}

In this section we prove Theorems~\ref{thm:HTS-implies-EVL},
\ref{thm:EVL=>HTS} and \ref{thm:cylinderEVL-HTS}.

\begin{proof}[Proof of Theorem~\ref{thm:HTS-implies-EVL}]  Set
\begin{align*}
u_n&=g_1\left(n^{-1}\right)
+p\left(g_1\left(n^{-1}\right)\right)
y,&&\mbox{for $y\in {\mathbb R}$, for type
$g_1$;}\\
 u_n&=g_2\left(n^{-1}\right)y,&&\mbox{for $y>0$, for type
$g_2$;}\\
u_n&=D-\left(D-g_3\left(n^{-1}\right)\right)(-y),
&&\mbox{for $y<0$, for type $g_3$.}
\end{align*}

For $n$ sufficiently large,
\begin{align}
\{x:M_n(x)\leq u_n\}&=\bigcap_{j=0}^{n-1}\{x:X_j(x)\leq u_n\}
=\bigcap_{j=0}^{n-1}\left\{x:g\left(\mu\left(B_{\dist(f^j(x),\zeta)}(\zeta)\right)\right)\leq u_n\right\}\nonumber\\
&=\bigcap_{j=0}^{n-1}\left\{x:\mu\left(B_{\dist(f^j(x),\zeta)}(\zeta)\right)\geq g^{-1}(u_n)\right\} \label{eq:rel-max-returns_1}.
\end{align}

Consequently, by (\ref{eq:l-property}),
\begin{align}
\mu(\{x:M_n(x)\leq u_n\})&=\mu\left(\bigcap_{j=0}^{n-1}\left\{x:\mu\{B_{\dist(f^j(x),\zeta)}
(\zeta)\}\geq \mu\{B_{\ell(g^{-1}(u_n))}
(\zeta)\} \right\}\right)\nonumber\\
&=\mu\left(\bigcap_{j=0}^{n-1}\left\{x:{\dist(f^j(x),\zeta)}
\geq {\ell(g^{-1}(u_n))}
 \right\}\right)\nonumber\\
 &=\mu\left(\left\{x:{r_{B_{\ell(g^{-1}(u_n))}(\zeta)}(x)}
\geq n\right\}\right).
\label{eq:rel-max-returns_2}
\end{align}

Now, observe that \eqref{eq:def-g1}, \eqref{eq:def-g2} and \eqref{eq:def-g3} imply
\begin{align*}
g_1^{-1}(u_n)&=g_1^{-1}\left[g_1(n^{-1})
+p\left(g_1(n^{-1})\right)
y\right]\sim g_1^{-1}
\left[g_1(n^{-1})\right]\e^{-y}
=\frac{\e^{-y}}{n};\\
g_2^{-1}(u_n)&=g_2^{-1}
\left[g_2(n^{-1})
y\right]\sim g_2^{-1}
\left[g_2(n^{-1})
\right]y^{-\beta}
=\frac{y^{-\beta}}{n};\\
g_3^{-1}(u_n)&=g_3^{-1}
\left[D-\left(D-g_3(n^{-1})\right)
(-y)\right]
\sim g_3^{-1}\left[
D-\left(D-g_3(n^{-1}\right)
\right] (-y)^{\gamma}=\frac{(-y)^{\gamma
}}{n}.
\end{align*}
Thus, we may write
$$g^{-1}(u_n)\sim\frac{\tau(y)}{n},$$
meaning that
$$g_i^{-1}(u_n)\sim\frac{\tau_i(y)}{n},
 \;\; \forall i\in \{1,2,3\}$$
where $\tau_1(y)=\e^{-y}$ for $y \in {\mathbb R}$, $\tau_2(y)=
y^{-\beta }$ for $y>0$, and $\tau_3(y)=(-y)^{\gamma }$ for $y<0$.

Recalling (\ref{eq:l-property}), we have
\[
\mu\left(B_{\ell(g^{-1}(u_n))}(\zeta)\right)\sim\frac{\tau(y)}{n},
\]
and so,
\begin{equation}
\label{eq:n-estimate}n\sim
\frac{\tau(y)}{\mu\left(B_{\ell(g^{-1}(u_n))}(\zeta)\right)}.
\end{equation}

 Now, we claim that using \eqref{eq:rel-max-returns_2} and
\eqref{eq:n-estimate}, we have
\begin{align}
\label{eq:approximation-1} \lim_{n\to\infty}\mu(M_n(x)\leq
u_n)&=\lim_{n\to\infty}
\mu\left(r_{B_{l(g^{-1}(u_n))}(\zeta)}(x)\geq
\frac{\tau(y)}{\mu\left(B_{\ell(g^{-1}(u_n))}(\zeta)\right)}\right)\\
&= G(\tau(y))\label{eq:conclusion-1},
\end{align}
which gives the first part of the theorem.

To see that \eqref{eq:approximation-1} holds, observe that by
\eqref{eq:rel-max-returns_2} and \eqref{eq:n-estimate} we have
\begin{multline*}
 \left|\mu(M_n\leq
u_n)-\mu\left(r_{B_{\ell(g^{-1}(u_n))}(\zeta)}\geq
\frac{\tau(y)}{\mu\left(B_{\ell(g^{-1}(u_n))}(\zeta)\right)}\right)
\right|
\\ =
\left| \mu\left(r_{B_{\ell(g^{-1}(u_n))}(\zeta)}\geq
n\right)-\mu\left(r_{B_{\ell(g^{-1}(u_n))}(\zeta)}\geq
(1+\varepsilon_n)n\right)\right|,
\end{multline*}
where $(\varepsilon_n)_{n\in\N}$ is such that $\varepsilon_n\to0$
as $n\to\infty$. Since we have
\begin{equation}
\label{eq:diff-r-m-k}\left\{r_{B_{\ell(g^{-1}(u_n))}(\zeta)}\geq
m\right\}\setminus\left\{r_{B_{\ell(g^{-1}(u_n))}(\zeta)}\geq
m+k\right\}\subset \bigcup_{j=m}^{m+k-1} f^{-j}\left(
B_{\ell(g^{-1}(u_n))}(\zeta)\right),\;\mbox{$\forall
m,k\in\N$,}\end{equation} it follows by stationarity that
\begin{multline*}
 \left|
\mu\left(r_{B_{\ell(g^{-1}(u_n))}(\zeta)}\geq
n\right)-\mu\left(r_{B_{\ell(g^{-1}(u_n))}(\zeta)}\geq
(1+\varepsilon_n)n\right)\right|\\ \leq |\varepsilon_n|n\mu\left(
B_{\ell(g^{-1}(u_n))}(\zeta)\right)\sim |\varepsilon_n|\tau \to 0,
\end{multline*}
as $n\to\infty$, completing the proof of \eqref{eq:approximation-1}.

Next we will use the exponential HTS hypothesis, that is
$G(t)=\e^{-t}$, to show the second part of the theorem.

Under the exponential HTS assumption, by \eqref{eq:conclusion-1}
it follows immediately that    $\mu(M_n(x)\leq
u_n)\to\e^{-\tau(y)}$, as $n\to\infty$.  Recall that in the corresponding i.i.d setting, i.e. when we are considering $\{x:\hat M_n(x)\leq u_n\}$ rather than $\{x: M_n(x)\leq u_n\}$, \eqref{eq:un} is equivalent to \eqref{eq:iid}.  Therefore we also have $\lim_{n\to\infty}\mu(\hat M_n(x)\leq u_n)=\e^{-\tau(y)}$, since $n\mu(X_0>u_n)=n\mu\left(B_{\ell(g^{-1}(u_n))}(\zeta)\right)\to\tau(y)$, as $n\to\infty$.
As explained in the introduction, this means that in the i.i.d. setting $G(\tau)$ must be of the three classical types.
It remains to show that if the observable is of type $g_i$ then $\lim_{n\to\infty}\mu(M_n(x)\leq
u_n)=\e^{-\tau(y)}$ means that the
EVL that applies to $M_n$ (rather than $\hat M_n$) is also of type $\ev_i$, for each
$i\in\{1,2,3\}$.

\textbf{Type $\mathbf{g_1}$:} In this case we have
$\e^{-\tau_1(y)}=\e^{-\e^{-y}}$, for all $y\in {\mathbb R}$, that
corresponds to the Gumbel evd and so we have an EVL for $M_n$ of
type $\ev_1$.

\textbf{Type $\mathbf{g_2}$:} We obtain $\e^{-\tau_2(y)}=\e^{-y^{-\beta}}$
for $y>0$.  To conclude that in this case we have the Fr\'{e}chet
evd with parameter $\beta$, we only have to check that for
$y\leq0$, $\mu(M_n(x)\leq u_n)=0$. Since
$g_2(n^{-1})>0$ (for
all large $n$) and $$\mu(M_n(x)\leq
u_n)=\mu\left(M_n(x)\leq
g_2(n^{-1}) y\right)\rightarrow\e^{-y^{-\beta }}$$ as
$n\to\infty$. Letting $y\downarrow 0$, it follows that
$\mu(M_n(x)\leq 0)\rightarrow 0$, and, for $y<0$,
$$\mu(M_n(x)\leq u_n)=\mu\left(M_n(x)\leq
g_2(n^{-1})y\right)\leq
\mu(M_n(x)\leq 0)\rightarrow 0.$$ So, we have, in this case, an EVL for $M_n$  of type $\ev_2$.

\textbf{Type $\mathbf{g_3}$:} For $y<0$, we have
$\e^{-\tau_3(y)}=\e^{-(-y)^{\gamma}}$.  To conclude that in this
case we have the Weibull evd with parameter $\gamma$, we only
need  to check that for $y\geq0$, $\mu(M_n(x)\leq u_n)=1$. In
fact, for $y\geq0$, since
$D-g_3(n^{-1})>0$, we
have \begin{align*} \mu(M_n(x)\leq u_n)&
=\mu\left(M_n(x)\leq
\left(D-g_3(n^{-1})
\right)y+D\right)\\
&\geq \mu(M_n(x)\leq D)=1.\end{align*} So we have, in this
case, an EVL for $M_n$  of type $\ev_3$.
\end{proof}

For the proof of Theorem~\ref{thm:EVL=>HTS}, we will require the following lemma.  This is essentially contained in \cite[Theorem~1.6.2]{LLR83}.  See also \cite[Lemma 2.1]{FFT10} where the lemma was proved for acips.  We provide a proof in the general case for completeness.
\begin{lemma}
Let $(\X,\mathcal B, \mu,f)$ be a
dynamical system, $\zeta\in\X$ and assume that $\mu$ is such that the function $\size$ defined on \eqref{eq:def-h} is continuous. Furthermore, let $\varphi$ be as in \eqref{eq:observable-form}.  Then, for each $y\in\R$, there exists a sequence $(u_n(y))_{n\in\N}$ as in \eqref{def:level-un} such that
$$n\mu(\{x:\varphi(x)>u_n(y)\}) \xrightarrow[n\to\infty]{} \tau(y)\geq 0.$$
Moreover, for every $t>0$ there exists $y\in\R$ such that $\tau(y)=t$.
\label{lem:t from tau}
\end{lemma}

\begin{proof}
We will prove the lemma in the case when $g$ is of type $g_2$.  For the other two types of $g$, the argument is the same, but with minor adjustments, see \cite[Theorem~1.6.2]{LLR83}.

First we show that we can always find a sequence $(\gamma_n)_{n\in\N}$ such that $$n\mu(X_0>\gamma_n)\xrightarrow[n\to\infty]{}1.$$ Take $\gamma_n:=\inf\{y:\mu(X_0\leq y)\geq 1-1/n\},$ and let us show that it has the desired property. Note that $n\mu(X_0>\gamma_n)\leq 1$, which means that $\limsup_{n\to \infty} n\mu(X_0>\gamma_n)\leq 1$. Using \eqref{eq:def-g2}, for any $z<1$, we have
\[
\liminf_{n\to\infty}\frac{\mu(X_0>\gamma_n)}{\mu(X_0>\gamma_n z)}=\liminf_{n\to\infty}\frac{\mu(B_{\ell(g_2^{-1}(\gamma_n))}(\zeta))}
{\mu(B_{\ell(g_2^{-1}(z\gamma_n))}(\zeta))}=
\liminf_{n\to\infty}\frac{{g_2^{-1}(\gamma_n)}}
{{g_2^{-1}(z\gamma_n)}}=z^{\beta},
\]
where $\ell$ is the function defined in (\ref{eq:l}).
Since, by definition of $\gamma_n$, for any $z<1$, $n\mu(X_0>\gamma_n z)\geq1$, letting $z\to 1$, it follows immediately that $\liminf_{n\to \infty} n\mu(X_0>\gamma_n)\geq 1$.

Now let $u_n(y)=\gamma_n y$, which means that, for all $n\in\N$, we are taking $a_n=\gamma_n^{-1}$ and $b_n=0$ in \eqref{def:level-un}. Then, using \eqref{eq:def-g2}, it follows that for all $y>0$
\begin{align*}
  n\mu(X_0>\gamma_n y)&=n\mu(B_{l(g_2^{-1}(\gamma_n y))}(\zeta))= n g_2^{-1}(\gamma_ny)\\
  &\sim n y^{-\beta}g_2^{-1}(\gamma_n)=
  y^{-\beta}n \mu(B_{\ell(g_2^{-1}(\gamma_n))}(\zeta))
  = y^{-\beta}n\mu(X_0>\gamma_n)\xrightarrow[n\to\infty]{}y^{-\beta}.
\end{align*}
So taking $y=t^{-1/\beta}>0$ would suit our purposes.
\end{proof}

\begin{proof}[Proof of Theorem~\ref{thm:EVL=>HTS}.]
We assume that by hypothesis for every $y\in\R$ and some sequence
$u_n=u_n(y)$ as in \eqref{def:level-un} such that $n \mu\left(\{x:\varphi(x)>
u_n(y)\}\right)\xrightarrow[n\to\infty]{}\tau(y)$,  we have
$$\lim_{n\to\infty}\mu\left(\{x:M_n(x)\leq
u_n(y)\}\right)=H(\tau(y)).$$
Observe that, by Khintchine's Theorem (see \cite[Theorem~1.2.3]{LLR83}), up to linear scaling the normalising sequences are unique, which means that we may assume that they are the ones given by Lemma~\ref{lem:t from tau}. Hence given $t>0$, Lemma~\ref{lem:t from tau} implies that there exists $y\in \R$ such that
$$n \mu\left(\{x:\varphi(x)>
u_n(y)\}\right)\xrightarrow[n\to\infty]{}t.$$
Given $(\delta_n)_{n\in\N}\subset \R^+$ with
$\delta_n\xrightarrow[n\to\infty]{}0$, we define $$\kappa_n:=\lfloor
t/\mu(B_{\delta_n}(\zeta))\rfloor.$$

We will prove
\begin{equation}
\label{eq:aux-1} g^{-1}\left(u_{\kappa_n}\right)\sim\mu(B_{\delta_n}(\zeta)).
\end{equation}

If $n$ is sufficiently large, then
\[
\left\{x:\varphi(x)>u_n\right\}=\left\{x:g(\mu(B_{\dist(x,\zeta)}(\zeta)))>u_n\right\}
=\left\{x:\mu(B_{\dist(x,\zeta)}(\zeta))<g^{-1} (u_n)\right\}.
\]

By \eqref{eq:l-property} and the definition of $\ell$ in \eqref{eq:l} we obtain
\begin{align*}
\mu(\{x:\varphi(x)>u_n\})&=\mu\left(\left\{x: \mu(B_{\dist(x,\zeta)}
(\zeta))< g^{-1}(u_n)\right\} \right)\\
&=\mu\left(\left\{x: \mu(B_{\dist(x,\zeta)}
(\zeta))< \mu(B_{\ell(g^{-1}(u_n))}
(\zeta))\right\} \right)\\
&=\mu\left(\left\{x:{\dist(x,\zeta)}
< {\ell(g^{-1}(u_n))}
 \right\}\right)\\
&=\mu\left( B_{\ell(g^{-1}(u_n))}(\zeta) \right).
\end{align*}

Hence, by assumption on the sequence $u_n$, we have
$n\mu\left(B_{\ell(g^{-1}(u_n))}(\zeta)\right)\xrightarrow[n\to\infty]{}
\tau(y)=t$. As we know that  $\mu\left( B_{\ell(g^{-1}(u_n))}(\zeta)\right)=g^{-1}(u_n)$, we have
$ng^{-1}(u_n)\xrightarrow[n\to\infty]{}
t$. Thus, we may write $g^{-1}(u_n)\sim \frac
t{n}$ and substituting $n$ by $\kappa_n$ we immediately
obtain \eqref{eq:aux-1} by definition of $\kappa_n$.

Again, by the definition of $\ell$ in \eqref{eq:l} and \eqref{eq:l-property} we note that
\begin{align}
\mu(\{x:M_{\kappa_n}(x)\leq u_{\kappa_n}\})&=\mu\left(\bigcap_{j=0}^{\kappa_n-1}\left\{x:\mu\{B_{\dist(f^j(x),\zeta)}
(\zeta)\}\geq g^{-1}(u_{\kappa_n}) \right\}\right)\nonumber\\
&=\mu\left(\bigcap_{j=0}^{{\kappa_n}-1}\left\{x:\mu\{B_{\dist(f^j(x),\zeta)}
(\zeta)\}\geq \mu\{B_{\ell(g^{-1}(u_{\kappa_n}))}
(\zeta)\} \right\}\right)\nonumber\\
&=\mu\left(\bigcap_{j=0}^{{\kappa_n}-1}\left\{x:{\dist(f^j(x),\zeta)}
\geq {\ell(g^{-1}(u_{\kappa_n}))}
 \right\}\right)\nonumber\\
 &=\mu\left(\left\{x:{r_{B_{\ell(g^{-1}(u_{\kappa_n}))}(\zeta)}(x)}
\geq {\kappa_n}\right\}\right).
\label{eq:Mln-rgeqln}
\end{align}

At this point, we claim that
\begin{equation}
\label{eq:approximation-2}
\lim_{n\to\infty}\mu\left(\left\{x:r_{B_{\delta_n}(\zeta)}(x)\geq\frac t
{\mu(B_{\delta_n}(\zeta))}\right\}\right)=\lim_{n\to\infty}
\mu(\{x:M_{\kappa_n}(x)\leq u_{\kappa_n}\}).
\end{equation}
Then, the first part of the theorem follows, since
by hypothesis,$$\mu\left(\{x:M_{\kappa_n}(x)\leq
u_{\kappa_n}\}\right)\xrightarrow[n\to\infty]{}H(\tau(y))=H(t).$$

For the second part of the theorem, first notice that for the iid setting, i.e. when we are considering $\{x:\hat M_n(x)\leq u_n\}$ rather than $\{x: M_n(x)\leq u_n\}$, \eqref{eq:un} is equivalent to \eqref{eq:iid}.  Therefore, $\mu(\{x:\hat M_n(x)\leq u_n\})\to \e^{-\tau(y)}$ as $n\to \infty$.  Hence if the EVL of $M_n$ coincides with that of $\hat M_n$, then we also have $H(\tau(y))=\e^{-\tau(y)}$.

It remains to show that \eqref{eq:approximation-2} holds. First, observe that
\begin{align*}
  \mu\left(r_{B_{\delta_n}(\zeta)}\geq\frac t
{\mu(B_{\delta_n}(\zeta))}\right)&=
\mu(M_{\kappa_n}\leq u_{\kappa_n})+
\left(\mu\left(r_{B_{\delta_n}(\zeta)}\geq \kappa_n\right)-
\mu(M_{\kappa_n}\leq u_{\kappa_n})\right)\\
&\quad+\left(\mu\left(r_{B_{\delta_n}(\zeta)}\geq\frac t
{\mu(B_{\delta_n}(\zeta))}\right)-\mu\left(r_{B_{\delta_n}(\zeta)}\geq
\kappa_n\right)\right).
\end{align*}

For the third term on the right, we note, by the definition of $\kappa_n$ that we have
\begin{multline*}
 \left|\mu\left(r_{B_{\delta_n}(\zeta)}\geq \kappa_n\right)-
 \mu\left(r_{B_{\delta_n}(\zeta)}\geq\frac t
{\mu(B_{\delta_n}(\zeta))}\right) \right|
\\ =
\left| \mu\left(r_{B_{\delta_n}(\zeta)}\geq
\kappa_n\right)- \mu\left(r_{B_{\delta_n}(\zeta)}\geq
(1+\varepsilon_n)\kappa_n\right)\right|,
\end{multline*}
for some sequence $(\varepsilon_n)_{n\in\N}$ such that
$\varepsilon_n\to0$, as $n\to\infty$. By \eqref{eq:diff-r-m-k} and stationarity it follows that
\begin{equation*}
 \left| \mu\left(r_{B_{\delta_n}(\zeta)}\geq \kappa_n\right)-
 \mu\left(r_{B_{\delta_n}(\zeta)}\geq (1+\varepsilon_n)\kappa_n\right)
 \right|
  \leq |\varepsilon_n|\kappa_n\mu(
B_{\delta_n}(\zeta))\sim|\varepsilon_n|t \to 0,
\end{equation*}
as $n\to\infty$.

For the remaining term, using the definition of $\kappa_n$ and
\eqref{eq:Mln-rgeqln}, we have
\begin{align*}
\left|\mu\left(\left\{r_{B_{\delta_n}(\zeta)}\geq \kappa_n\right\}\right)-\mu\left(\{M_{\kappa_n}\leq u_{\kappa_n}\}\right)\right|&=
\left|\mu\left(\left\{r_{B_{\delta_n}(\zeta)}\geq \kappa_n\right\}\right)-\mu\left(\{r_{B_{\ell(g^{-1}(u_{\kappa_n}))}(\zeta)}\geq \kappa_n\}\right)\right|\\
&\leq \sum_{i=1}^{\kappa_n}\mu\left(f^{-i}\left(B_{\delta_n}(\zeta)\bigtriangleup
B_{\ell(g^{-1}(u_{\kappa_n}))}(\zeta)\right)\right)\\
&=\kappa_n\mu\left(B_{\delta_n}(\zeta)\bigtriangleup
B_{\ell(g^{-1}(u_{\kappa_n}))}(\zeta)\right)\\
&\sim \frac t{\mu\left(B_{\delta_n}(\zeta)\right)}\left|\mu\left(B_{\delta_n}(\zeta)
\right)-
\mu\left(B_{\ell(g^{-1}(u_{\kappa_n}))}(\zeta)\right)\right|\\
&=t\left|1-\frac{\mu\left(B_{\ell(g^{-1}(u_{\kappa_n}))}(\zeta)\right)}{
\mu\left(B_{\delta_n}(\zeta)\right)}\right|,
\end{align*}
which, by (\ref{eq:l-property}) and (\ref{eq:aux-1}), tends to $0$ as $n\to\infty$; this ends the proof of \eqref{eq:approximation-2}.
\end{proof}

\begin{proof}[Proof of Theorem~\ref{thm:cylinderEVL-HTS}]  For every $n\in\N$, set
$
u_n=g\left(\mu\left(\cyl_{n-1}[\zeta]\right)\right).
$
Then, by definition of $\psi$, it follows immediately that, for large $n\in\N$
\begin{align*}
\{X_0>u_n\}&=\left\{x: \psi(x)<g^{-1}(u_n)\right\}=\left\{x: \psi(x)<\mu\left(\cyl_{n-1}[\zeta]\right)\right\}=\cyl_n[\zeta],
\end{align*}
which means that condition (\ref{eq:un-cylinders}) is verified. Let $\omega_n$ be defined by (\ref{eq:wn-definition}). It is also clear that
\begin{align*}
\{M_{\omega_n}\leq u_n\}&=\bigcap_{j=0}^{\omega_n-1}\{x: g\left(\psi\left(f^j(x)\right)\right)\leq g^{-1}(u_n)\}=
\bigcap_{j=0}^{\omega_n-1} \left\{x: \psi(f^j(x))<g^{-1}(u_n)\right\}\\
&=\bigcap_{j=0}^{\omega_n-1} \left\{x: \psi(f^j(x))\geq\mu\left(\cyl_{n-1}[\zeta]\right)\right\}=\bigcap_{j=0}^{\omega_n-1}\left\{x: f^j(x)\notin\cyl_n[\zeta] \right\}=\left\{r_{\cyl_n[\zeta]}\geq\omega_n\right\}.
\end{align*}
Since $\left|\omega_n-\frac t{\mu(\cyl_n[\zeta])}\right|\leq 1$, recalling \eqref{eq:diff-r-m-k} and using stationarity  we have
\begin{align*}
\left|\mu(M_{\omega_n}\leq u_n)-\mu\left(r_{\cyl_n[\zeta]}\geq\frac t{\mu(\cyl_n[\zeta])}\right)\right|&=\left|\mu(r_{\cyl_n[\zeta]}\geq\omega_n)-\mu\left(r_{\cyl_n[\zeta]}\geq\frac t{\mu(\cyl_n[\zeta])}\right)\right|\\
&\leq\mu(\cyl_n[\zeta])\xrightarrow[n\to\infty]{}0.
\end{align*}

Now, the result follows at once.
\end{proof}

\section{Applications}
\label{sec:aps}

In this section we describe the types of systems that Theorems~\ref{thm:HTS-implies-EVL}-\ref{thm:cylinderEVL-HTS} apply to, specifically when the measures are equilibrium states.  Thus the equivalences given in those theorems yield new EVLs and HTS in those settings.

\subsection{Equilibrium states and SRB measures}
\label{sec:eq}

For a discrete time dynamical system $f:X \to X$, we let $$\M=\M(f):=\left\{\text{measures }\mu: \mu\circ f^{-1}=\mu \text{ and } \mu(X)=1\right\}.$$  Given a \emph{potential} $\phi: X \to [-\infty, \infty]$, the \emph{pressure} of $\phi$ with respect to $f$ is defined as
\begin{equation*}
P(\phi) =P(f,\phi):= \sup \left\{ h_\mu + \int \phi~d\mu :  \mu \in \M \textrm{ and } - \int \phi~d\mu < \infty\right\},
\end{equation*}
where $ h_\mu$ denotes the measure theoretic entropy of $f$ with respect to $\mu$.   
A measure $\mu\in \M$ which `achieves the pressure', i.e. with $h_\mu+\int\phi~d\mu=P(\phi)$, is called an \emph{equilibrium state}.
For example, if we set $\phi$ to be a constant, then the relevant equilibrium state is the measure of maximal entropy.

If $\P_1$ is a partition of $X$ and we refine the partition to obtain $\P_n = \bigvee_{i=0}^{n-1} f^{-i} \P_1$ as above.  Let $S_n \phi(x) := \sum_{k=0}^{n-1} \phi \circ f^k(x)$ be the $n$-th ergodic
sum along the orbit of $x$.
We say that $\mu$ satisfies the \emph{Gibbs property}
if for $\mu$-a.e. $x\in X$ there are $K,P\in \R$ such that
\begin{equation}\label{eq:Gibbs}
\frac{1}{K} \le \frac{\mu(\cyl_n[x])}{e^{S_n \phi(x)-nP}}
\le K.
\end{equation}

\begin{remark}
Notice that \eqref{eq:Gibbs} implies that the $\mu$-measure of $n$-cylinders around a $\mu$-typical point $x$ fluctuates as much as $S_n\phi(x)$ fluctuates.  In general for non-constant potentials, the Law of the Iterated Logarithm (see \cite{DP84}) implies that $\frac{S_n\phi(x)}n$ has liminf and limsup equal to 0 and $\infty$ respectively.
Observe that the potential $\phi$ for the full tent map in Section~\ref{sec:cyl} is the constant $\log 2$.
\label{rmk:Gibbs fluc}
\end{remark}

\emph{Sinai-Ruelle-Bowen measures} (SRB measures) are often used to analyse dissipative chaotic systems. Assume that
$f:X \to X$ is a $C^2$
diffeomorphism of a finite dimensional manifold $X$ with a volume form defined on the Borel sets of $X$ that we call Lebesgue measure. We have in mind dissipative systems that present chaotic strange attractors such as the families of Lozi maps \cite{L78} or H\'enon maps \cite{H76}. The fact that these systems contract volume rules out the possibility of invariant measures equivalent to Lebesgue. SRB measures are like the next best thing when no invariant measure equivalent to volume exists. Their relation with volume is that conditional to unstable manifolds they are absolutely continuous with respect to the conditional Lebesgue measure on those leaves.
We do not give a formal definition of these measures: instead we refer the reader to \cite{Y02} which contains a very complete description.
However, we would like to emphasise that these measures, at least in the examples we mention here, are not absolutely continuous with respect to Lebesgue measure, do not have atoms and can even be realised as equilibrium states for a certain potentials.  Therefore the current paper provides tools to understand HTS and EVLs for these measures.



%

\subsection{Particular systems to which Theorems~\ref{thm:HTS-implies-EVL}-\ref{thm:cylinderEVL-HTS} apply}

Exponential HTS to cylinders have been shown for many hyperbolic systems.  This was first shown for Axiom A maps $f:X\to X$ with equilibrium state $\mu_\phi$ with respect to a H\"older potential $\phi:X\to \R$ in \cite{H93} (see also \cite{P91} for the Markov chain setting).  In this case the equilibrium state $\mu_\phi$ is a Gibbs state which satisfies a mixing condition called `$\alpha$-mixing', for details see \cite{HSV99} and \cite{AG01}.  A theory for HTS to cylinders for various dynamical systems with Gibbs states with various mixing conditions, such as $\alpha$-mixing, can be found in for example \cite{AG01, HSV99, AS10}.  One of the issues of interest in these cases is the rate of convergence to the exponential law.

So in all of the above cases we can apply Theorem~\ref{thm:cylinderEVL-HTS} to get EVLs for cylinders.

The problem of proving HTS to \emph{balls} in dimension higher than one is often complicated by the fact that the measure of small balls can behave badly as the size of the ball shrinks. 
Mainly because of this, much more is known for HTS to balls in one dimension.  For example if $f:I\to I$ is a multimodal map of the unit interval $I$ and $\phi=-t\log|Df|$ is a potential with an equilibrium state $\mu_t$ (see \cite{IT09} for the most general result in on existence of such equilibrium states), it was proved in \cite{BT09a} that the system has exponential HTS to balls.  A similar result was proved for H\"older potentials $\phi:I \to \R$ with $\phi<P(\phi)$.  These results also hold for Manneville-Pomeau maps, see for example \cite{BST03}.  Then Theorem~\ref{thm:HTS-implies-EVL} applies to all these cases to obtain EVLs.

In \cite{GNH09} the authors considered among other systems the family of Lozi maps, which have SRB measures $\mu$ that are not absolutely continuous with respect to Lebesgue, and prove the existence of EVLs of the classical type for $\mu$-a.e. point in the support of $\mu$. In this setting, we can apply Theorem~\ref{thm:EVL=>HTS} to obtain HTS for Lozi maps.

\subsection{Non exponential laws}

In \cite{CF96} it is shown that for an irrational circle rotation $(S^1, f)$ with Lebesgue measure there are subsequences $(n_i)_{i\in \N}$ and $(k_j)_{j\in \N}$ such that the HTS to nested cylinders along these subsequences are $G_1$ and $G_2$ respectively.  Here $G_1$ comes from the uniform distribution and $G_2$ is a particular piecewise linear function, so neither of the distributions are exponential.  Note that $(S^1,f)$ has zero topological entropy.  So in this case we can apply a modified version of Theorem~\ref{thm:cylinderEVL-HTS}, replacing the sequence $(u_n)_{n\in \N}$ in \eqref{eq:un-cylinders} with
$(u_{n_i})_{i\in \N}$ or $(u_{k_j})_{j\in \N}$ and making the corresponding changes to $(\omega_n)_{n\in \N}$ in \eqref{eq:wn-definition}.  This yields EVLs other than types 1-3.

\emph{Acknowledgements.}  
We would like to thank P.\ Varandas for encouragement and for useful comments. We are also obliged to M.\ Nicol for helpful suggestions.



\end{document}